\documentclass{article}

\usepackage{amsmath,amsthm,amssymb}
\usepackage{mathrsfs,bm}
\usepackage{framed}
\usepackage{mathtools, enumitem}
\usepackage{cases}

\usepackage{hyperref}
\hypersetup{%
colorlinks=true,
linkcolor=blue
}

\usepackage{tikz}

\newcommand{\e}{\mathrm{e}}
\newcommand{\iu}{\mathrm{i}}
\newcommand{\df}{\mathrm{d}}

\newcommand{\lap}{\Delta}

\newcommand{\ds}{\displaystyle}

\makeatletter
\newcommand{\leqnomode}{\tagsleft@true\let\veqno\@@leqno}
\makeatother

\mathtoolsset{showonlyrefs=true}

\usepackage{expl3}
\ExplSyntaxOn
\cs_new_eq:NN \Repeat \prg_replicate:nn
\ExplSyntaxOff

\newcommand{\jbk}[1]{\left\langle {#1} \right\rangle}

\newcommand{\rmop}[1]{\mathop{\mathrm{#1}}}

\newcommand{\beq}{\begin{equation}}
\newcommand{\eeq}{\end{equation}}
\def\ol{\overline}

\newcommand{\nv}{\nu}

\renewcommand{\Re}{\rmop{Re}}

\newcommand{\p}{\partial}

\newcommand{\Cbb}{\mathbb{C}}

\newcommand{\Kbb}{\mathbb{K}}

\newcommand{\Nbb}{\mathbb{N}}

\newcommand{\Rbb}{\mathbb{R}}

\newcommand{\Ga}{\alpha}
\newcommand{\Gb}{\beta}

\newcommand{\Ge}{\epsilon}

\newcommand{\Gg}{\gamma}

\newcommand{\Gl}{\lambda}

\newcommand{\Gv}{\nu}

\newcommand{\Gr}{\rho}

\newcommand{\Gs}{\sigma}

\newcommand{\Gj}{\tau}

\newcommand{\Gz}{\zeta}
\newcommand{\GD}{\Delta}

\newcommand{\GO}{\Omega}

\newtheorem{prop}{Proposition}[section]
\newtheorem{theo}[prop]{Theorem}
\newtheorem{coro}[prop]{Corollary}
\newtheorem{lemm}[prop]{Lemma}

\theoremstyle{definition}

\newtheorem*{note*}{Note}
\newtheorem*{claim*}{Claim}
\newtheorem*{exam*}{Example}
\newtheorem*{rema*}{Remark}
\newtheorem*{exer*}{Exercise}
\newtheorem*{prob*}{Problem}

\numberwithin{equation}{section}

\begin{document}

\title{Finiteness of the stress in presence of closely located inclusions with imperfect bonding\thanks{This work was supported by NRF of S. Korea grant no. 2022R1A2B5B01001445, KIAS Individual grant no. MG089001 at Korea Institute for Advanced Study, and NSF of China grant no. 11901523.}}

\author{Shota Fukushima\thanks{Department of Mathematics and Institute of Applied Mathematics, Inha University, Incheon 22212, S. Korea. Email: \texttt{shota.fukushima.math@gmail.com}.} \and Yong-Gwan Ji\thanks{School of Mathematics, Korea Institute for Advanced Study, Seoul 02455, S. Korea. Email: \texttt{ygji@kias.re.kr}.} \and Hyeonbae Kang\thanks{Department of Mathematics and Institute of Applied Mathematics, Inha University, Incheon 22212, S. Korea. Email: \texttt{hbkang@inha.ac.kr}.} \and Xiaofei Li\thanks{College of Science, Zhejiang University of Technology, Hangzhou, 310023, P. R. China. Email: \texttt{xiaofeili@zjut.edu.cn}.}}

\maketitle

\begin{abstract}
If two conducting or insulating inclusions are closely located, the gradient of the solution may become arbitrarily large as the distance between inclusions tends to zero, resulting in high concentration of stress in between two inclusions. This happens if the bonding of the inclusions and the matrix is perfect, meaning that the potential and flux are continuous across the interface. In this paper, we consider the case when the bonding is imperfect. We consider the case when there are two circular inclusions of the same radii with the imperfect bonding interfaces and  prove that the gradient of the solution is bounded regardless of the distance between inclusions if the bonding parameter is finite. This result is of particular importance since the imperfect bonding interface condition is an approximation of the membrane structure of biological inclusions such as biological cells.
\end{abstract}



\tableofcontents

\section{Introduction}

The purpose of this paper is to quantitatively investigate the stress occurring when two inclusions are closely located in a composite. The stress is represented by the gradient of the solution to the problem of a relevant partial differential equation. As we will review shortly, it is well understood that if the inclusions and the background matrix are perfectly bonded, then the stress can be arbitrarily large as the distance between two inclusions tends to zero. The perfect bonding is characterized by the continuities of the potential and the flux along the interfaces. The question of this paper is what happens if the bonding of the inclusions and the matrix is imperfect, i.e., one of the continuities of the potential and the flux fails to hold. The motivation to ask such a question is related to the stress in biological systems. In biological systems, inclusions, such as cells, often take forms of membrane consisting of the core and the shell and such membranes are bonded to the matrix. The membrane structure is approximated by the imperfect bonding condition by sending the width of the shell to zero as shown in \cite{BM1999}. Thus the question is whether the stress is reduced to finite with the imperfect bonding.

It is quite interesting to observe that neutral inclusions provide such inclusions with finite stress. If we arrange a single circular inclusion with a specially chosen imperfect bonding parameter or a membrane of the concentric core-shell with specially chosen conductivity, then the inclusion becomes neutral to uniform fields, i.e., the insertion of the inclusion does not perturb the uniform fields (\cite{Hashin62, Torquato-Rintoul95}). Even if we insert two of them, the uniform field is not perturbed. Thus the gradient of the solution is not perturbed, in particular bounded, regardless of the distance between such inclusions. We emphasize that a uniform field can cause blow-up of stress between inclusions with perfect bonding.

We consider in this paper the geometrically simplest possible case when inclusions are disks of the same radii and prove that the stress is bounded regardless of the distance between two inclusions provided that the imperfect bonding parameter is finite.

Let us now put things in precise terms. Let $D_1$ and $D_2$ be bounded simply connected domains in $\Rbb^d$ ($d\geq 2$) whose closures are disjoint. We assume that they are strictly convex and have smooth boundaries for convenience even though the assumptions can be relaxed slightly. We assume that the conductivity $k_m$ of the matrix $\Rbb^d \setminus \ol{D_1 \cup D_2}$ is 1, while that of the inclusions is $k \neq 1$. Since we will be considering the case when $k=\infty$ or $k=0$, we assume that inclusions have the same conductivity. We consider the following conductivity problem:
\begin{equation}\leqnomode
\label{PB}\tag{PB}
\begin{cases}
\GD u = 0  \quad&\mbox{in } \Rbb^d \setminus (\p D_1 \cup \p D_2), \\
\ds u|_+ - u|_- = 0  \quad&\mbox{on }\p D_j, \ j=1,2 , \\
\ds \p_\nu u|_+ - k \p_\nu u|_- = 0  \quad&\mbox{on }\p D_j, \ j=1,2 ,\\
\ds u (x) - h(x) = O( |x|^{1-d})~&\mbox{as }|x|\rightarrow\infty.
\end{cases}
\end{equation}
Here and throughout this paper, $\p_\nu$ denotes the outward normal derivative on $\p D_j$ and the subscripts $\pm$ denote the limits (along the normal direction) from outside and inside of $D_j$, respectively, and $h$ is a given function harmonic in $\Rbb^d$ (it represents the loading at $\infty$). The second and third lines in the problem \eqref{PB} are continuity of the potential and the flux, respectively, which characterize the perfect bonding conditions along $\p D_j$.

We are interested in the gradient $\nabla u$ of the solution $u$ to \eqref{PB} since it represents the stress. If $k$ is finite, then $\nabla u$ is bounded as proved in \cite{LN, LV}. However, if $k$ degenerates to either $\infty$ or $0$, then $\nabla u$, namely, the stress may become arbitrarily large as the distance $\Ge$ between $D_1$ and $D_2$ tends to zero depending on given $h$. In the conducting case, namely, when $k=\infty$, \eqref{PB} is reduced to the following problem (C for conducting):
\begin{equation}\leqnomode
    \label{PBC}\tag{PB-C}
\begin{cases}
\GD u = 0  \quad&\mbox{in } \Rbb^d \setminus \ol{D_1 \cup D_2}, \\
u|_+=\text{const.} \quad&\mbox{on }\p D_j, \ j=1,2 , \\
\ds \int_{\p D_j} \p_\nu u|_+ \, \df \Gs=0 \quad&\mbox{on }\p D_j, \ j=1,2 ,\\
\ds u (x) - h(x) = O( |x|^{1-d})~&\mbox{as }|x|\rightarrow\infty.
\end{cases}
\end{equation}
The constant values of $u|_+$ on $\p D_j$ are not given, but determined by the third condition. It is worthwhile to mention that the values on $\p D_1$ and $\p D_2$ are different in general.  This problem has been well-studied. It is shown that $\nabla u$ may blow up as $\Ge$ tends to $0$ depending on the given harmonic function $h$ and the optimal blow-up rate is $\Ge^{-1/2}$ in two-dimensions \cite{AKL, Keller-JAP-63, Yun}, $(\Ge |\ln\Ge|)^{-1}$ in three-dimensions \cite{BLY} (see also \cite{LY-CPDE-09} for the case of spherical inclusions). If $k=0$, namely, inclusions are insulated, then \eqref{PB} is reduced to the following problem (I for insulating):
\begin{equation}\leqnomode
\label{PBI}\tag{PB-I}
\begin{cases}
\GD u = 0  \quad&\mbox{in } \Rbb^d \setminus \ol{D_1 \cup D_2}, \\
\p_\nu u|_+ =0 \quad&\mbox{on }\p D_j, \ j=1,2 , \\
\ds u (x) - h(x) = O( |x|^{1-d})~&\mbox{as }|x|\rightarrow\infty.
\end{cases}
\end{equation}
The insulating case in two-dimensions is the dual problem to the conducting case, and the optimal rate of blow-up of $\nabla u$ is the same as that of conducting case (see, for example, \cite{AKL}). However, the three-dimensional case has not been solved for a while. Quite recently significant progress has been made and an optimal blow-up rate is obtained in \cite{DLY22, DLY21} (see also \cite{Yun16}). We mention that more general estimates, including higher order derivatives, obtained in \cite{DL} for circular inclusions. There are huge work related with this area of research. We refer to references in \cite{Kang23} for them.

The question we raise in this paper is what happens in a biological system. Biological inclusions such as cells take the form of a membrane consisting of a core and a shell (Figure \ref{fg_core_shell}).
So their bonding to the matrix is different from the perfect bonding by the single interface. The question we address is whether large stress can concentrate in between two closely located biological inclusions.
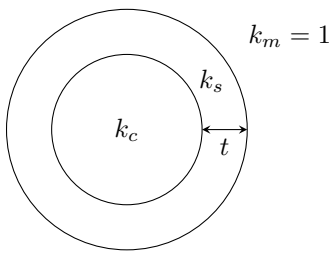
\begin{figure}[htb]
    \centering
    \begin{tikzpicture}[scale=1.0]
        \draw (0, 0) circle [radius=10mm];
        \draw (0, 0) circle [radius=16mm];
        \node (core) at (0, 0) {$k_c$};
        \node (shell) at (30:13mm) {$k_s$};
        \node (matrix) at (30:25mm) {$k_m=1$};
        \draw[<->, >=stealth] (10mm, 0)--(16mm, 0);
        \node (width) at (13mm, 0) [below] {$t$};
    \end{tikzpicture}
    \caption{Membrane of core-shell structure: $k_c, k_s, k_m$ are conductivities of the core, shell, and matrix, respectively, $t$ is the width of the shell.}
    \label{fg_core_shell}
\end{figure}

We do not directly deal with the core-shell structure. Instead we deal with approximations of such structures obtained by sending the width of the shell to $0$ which is derived in  \cite{BM1999}. To explain approximations let $\GO$ be an inclusion of the core-shell structure. Denote the conductivities of the core and the shell by $k_c$ and $k_s$, respectively, and the width of the shell by $t$ (Figure \ref{fg_core_shell}). Note that the core-shell structure has two interfaces: the boundary of the core and the outer boundary of the shell. Each interface satisfies the perfect bonding condition. Now we send the thickness $t$ of the shell to $0$. If $\Gg^{-1}:=\lim_{t \to 0} k_s/t$ exists and is positive, then the interface conditions (the second and third lines in \eqref{PB}) are changed to the following condition:
\beq
\p_\Gv u|_+ = k \p_\Gv u|_- = \Gg^{-1} \left(u|_+ - u|_- \right) \text{ on } \p \GO,
\eeq
where $k=k_c$. This is a kind of a Robin-type interface condition and shows that the continuity of the flux  still holds on the limiting single interface, but the continuity of the potential fails to hold. So, this is called an imperfect bonding in comparison to the perfect bonding. Since the limit of $k_s/t$ exists, $k_s$ needs to tend to $0$, and hence this approximation is called a LC-type (Low Conductivity).
If $\Ga:= \lim_{t \to 0} k_s t$ exists and $\Ga >0$, then it is shown in the same paper that the interface conditions approach to the conditions $u|_+ = u|_-$ and
\beq
k \p_\Gv u|_- - \p_\Gv u|_+  = \Ga \GD_S u + \nabla_S u \cdot \nabla_S \Ga \ \text{ on } \p \GO,
\eeq
where $\GD_S$ and $\nabla_S$ are the surface Laplacian and gradient on $\p\GO$, respectively. So, the potential stays continuous, but the flux becomes discontinuous. This is called HC-type (High Conductivity) approximation since $k_s$ needs to tend to $\infty$. We emphasize that the parameters $\Gg, \Ga$ may be variables even though we only consider constant parameters in this paper.

The imperfect bonding problems with the two inclusions, $D_1$ and $D_2$, are formulated as follows:
the LC-type problem is
\begin{equation}\leqnomode
\label{LC}\tag{LC} 
\begin{cases}
\GD u = 0 &\text{ in } \Rbb^d \setminus (\p D_1 \cup \p D_2), \\
\p_\Gv u|_+ = k \p_\Gv u|_- = \Gg^{-1} \left(u|_+ - u|_- \right) &\text{ on } \p D_j \ (j=1,2), \\
u(x) - h(x)  = O(|x|^{1-d}) &\text{ as } |x| \to \infty,
\end{cases}
\end{equation}
and the HC-type problem is
\begin{equation}\leqnomode
\label{HC}\tag{HC} 
\begin{cases}
\GD u = 0 &\text{ in } \Rbb^d \setminus (\p D_1 \cup \p D_2), \\
u|_+ = u|_- &\text{ on } \p D_j \ (j=1,2), \\
k \p_\Gv u|_- - \p_\Gv u|_+  = \Ga \GD_S u + \nabla_S u \cdot \nabla_S \Ga &\text{ on } \p D_j \ (j=1,2), \\
u(x) - h(x)  = O(|x|^{1-d}) &\text{ as } |x| \to \infty.
\end{cases}
\end{equation}
The problem is to estimate the gradient $\nabla u$ of the solution in terms of the distance $\Ge$ between $D_1$ and $D_2$. In comparison to the perfect bonding problems, we are particularly interested in the cases when $k=\infty$ for \eqref{LC} and $k=0$ for \eqref{HC}. Since presumably a biological system cannot endure large stress, we conjecture that $\nabla u$ is bounded regardless of $\Ge$ provided that $\Gg$ and $\Ga$ stay away from $0$. The purpose of this paper is to prove the conjecture in a special case.

We deal with the case when $D_1$ and $D_2$ are two-dimensional disks of the same radii. We also assume that $k=\infty$ and $\Gg$ is constant for \eqref{LC}, and that $k=0$ and $\Ga$ is constant for \eqref{HC}. Since $u|_-$ is constant, the \eqref{LC} becomes
\begin{equation}\leqnomode
\label{LCC}\tag{LC-C}
\begin{cases}
\GD u = 0 &\text{in } \Rbb^2 \setminus \ol{(D_1 \cup D_2)}, \\
u|_+ - \Gg\p_\Gv u|_+ = \text{const.} &\text{on } \p D_j \ (j=1,2), \\
\ds \int_{\p D_j} \p_\nu u|_+ \, \df \Gs=0 & j=1, 2, \\
u(x) - h(x)  = O(|x|^{-1}) &\text{as } |x| \to \infty.
\end{cases}
\end{equation}
If $\Gg=0$, then the problem becomes \eqref{PBC}. On the other hand, \eqref{HC} becomes
\begin{equation}\leqnomode
\label{HCI}\tag{HC-I} \
\begin{cases}
\GD u = 0 &\text{ in } \Rbb^2 \setminus \ol{(D_1 \cup D_2)}, \\
- \p_\Gv u|_+  = \Ga \GD_S u  &\text{ on } \p D_j \ (j=1,2), \\
u(x) - h(x)  = O(|x|^{-1}) &\text{ as } |x| \to \infty.
\end{cases}
\end{equation}
If $\Ga=0$, then this problem becomes \eqref{PBI}.
We further assume that the given harmonic function $h$ is linear. This is to compare the results of this paper with the known results for the perfect-bonding problem. It is proved in \cite{AKLLL} that if $h(x)=a \cdot x$ for some constant vector $a$ which is not perpendicular to the line connecting centers of two circular inclusions, then the gradient of the solution to \eqref{PBC} (and \eqref{PBI}) blows up at the rate of $\Ge^{-1/2}$. So, we set
\beq\label{D1D2}
D_1=B_r (-r-\Ge/2, 0) \quad \mbox{and} \quad D_2=B_r (r+\Ge/2, 0),
\eeq
after a rotation and a translation if necessary, where $B_r(a)$ denotes the disk of radius $r$ centered at $a$. We then let $h(x_1,x_2)=x_1$ for \eqref{LCC} and $h(x_1,x_2)=x_2$ for \eqref{HCI}.

We now introduce a function space where the solution to \eqref{LCC} uniquely exists. Let $D=D_1\cup D_2$, and define the local Sobolev space $H_\mathrm{loc}^1 (\Rbb^2\setminus D)$ (not $H_\mathrm{loc}^1 (\Rbb^2\setminus \overline{D})$) to be the collection of $u\in L^1_\mathrm{loc}(\Rbb^2\setminus D)$ such that $u|_{O\setminus \overline{D}}$ belongs to the $L^2$-Sobolev space of exponent $1$ on $O\setminus \overline{D}$ for any open disk $O\subset \Rbb^2$ including $\overline{D}$. We denote by $h+H_\mathrm{loc}^1 (\Rbb^2\setminus D)$ the space of all functions which can be represented as a sum of the function $h$ and a function in $H^1_\mathrm{loc}(\Rbb^2\setminus D)$.

The main results of this paper for \eqref{LCC} are the following.
\begin{theo}\label{mainLC}
Let $D_1$ and $D_2$ be as in \eqref{D1D2} and let $h(x_1,x_2)=x_1$. For each $\Gg>0$ and $\Ge >0$ there exists a unique solution $u \in h+H^1_\mathrm{loc} (\Rbb^2\setminus D)$ to \eqref{LCC}. Moreover, for each $\Gg_0>0$ and $\Ge_0>0$, there is a constant $C$ independent of $\Ge\in (0, \Ge_0]$ such that
    \beq\label{LC-C_gradient_estimate}
        \|\nabla (u-h)\|_{L^\infty (\Rbb^2\setminus \overline{D})}\leq C
    \end{equation}
    for all $\Gg \ge \Gg_0$.
\end{theo}

We obtain the following theorem for the problem \eqref{HCI}.

\begin{theo}\label{mainHC}
Let $D_1$ and $D_2$ be as in \eqref{D1D2} and let $h(x_1,x_2)=x_2$. For each $\Ga>0$ and $\Ge >0$ there exists a unique solution $u \in h+H^1_\mathrm{loc}(\Rbb^2\setminus D)$ to \eqref{HCI}. Moreover, for each $\Ga_0>0$ and $\Ge_0>0$, there is a constant $C$ independent of $\Ge\in (0, \Ge_0]$ such that
    \beq\label{HC-C_gradient_estimate}
        \|\nabla (u-h)\|_{L^\infty (\Rbb^2\setminus \overline{D})}\leq C
    \end{equation}
    for all $\Ga \ge \Ga_0$.
\end{theo}

The rest of the paper is devoted to proofs of main results. Theorem \ref{mainLC} is proved by constructing solutions in terms of infinite series and estimating the coefficients. For that purpose, we use the discrete Laplacian in an essential way. Theorem \ref{mainHC} is proved by showing duality of the problems \eqref{LCC} and \eqref{HCI}.

\section{Proof of Theorem \ref{mainLC}}\label{sect_discs}

Let us first show that the solution to \eqref{LCC} is unique. If $u \in H^1_\mathrm{loc}(\Rbb^2\setminus D)$ is the solution to \eqref{LCC} with $h=0$, then we have from the second and third conditions that
\beq\label{1000}
\int_{\p D_j} \p_\nu u|_+ (u|_+ - \Gg\p_\Gv u|_+) \, \df \Gs=0.
\eeq
It can be shown easily that 
\beq\label{1001}
|\nabla u(x)|=O(|x|^{-2}) \quad\text{as } |x|\to \infty.
\eeq
In fact, since $u$ is harmonic in $\Rbb^2 \setminus \overline{B}$ where $B$ an open disk centered at $0$ containing $\ol{D}$, there exist an entire function $f(z)$ and a holomorphic function $g(z)$ on $\Cbb\setminus \overline{B}$ with $g(z)\to 0$ as $|z|\to \infty$, and a constant $C\in \Rbb$ such that
    \begin{equation}
        \label{eq_Lap_ext_expansion_2d}
        u (z)=\Re (f(z)+g(z))+C\ln |z|, \quad z\in \Cbb \setminus \overline{B}.
    \end{equation}
(See, for example, \cite[Chapter 9]{ABR01}.) Since $u(z)=O(|z|^{-1})$ as $|z|\to \infty$, we infer that $f=0$ and $C=0$. So \eqref{1001} follows. It then follows from \eqref{1000} and \eqref{1001} that 
$$
- \int_{\Rbb^2 \setminus D} |\nabla u|^2 \, \df x - \Gg \sum_{j=1}^2 \int_{\p D_j} |\p_\nu u|_+|^2  \, \df \Gs=0.
$$
Since $\Gg >0$, we have $u=0$.

In the rest of this section we construct the solution and estimate its gradient.

\subsection{M\"obius transform}

We first introduce two quantities which play essential roles throughout this paper. Let
\begin{equation}\label{eq_rho_defi}
    \Gr:=\frac{\sqrt{4r+\Ge}-\sqrt{\Ge}}{\sqrt{4r+\Ge}+\sqrt{\Ge}}
\end{equation}
and
\begin{equation}\label{eq_beta}
    \Gb:=\sqrt{\Ge (4r+\Ge)}.
\end{equation}
Following \cite{Ji-Kang22}, we consider the M\"obius transform
\begin{equation}\label{eq_Moebius}
    T (z):=\frac{\Gb}{z-\Gb/2}+1, \quad z\in \widehat{\Cbb}:=\Cbb\cup \{\infty\}.
\end{equation}
Then one can easily see that
\begin{align*}
    T (D_1)=B_{\Gr}(0), \quad
    T (D_2)=\widehat{\Cbb}\setminus \overline{B_{\Gr^{-1}} (0)}.
\end{align*}

Let
\beq
A:=B_{\Gr^{-1}}(0)\setminus \overline{B_{\Gr}(0)} = T(\widehat{\Cbb} \setminus \overline{D}).
\eeq
For a solution $u$ to \eqref{LCC} let
$$
U(\Gz) = (u \circ T^{-1})(\Gz), \quad  \Gz \in A.
$$
We then look into conditions which $U$ needs to satisfy.

Let $\Gz = T(z)$. The relationship between the outward normal $\Gv_\Gz$ on $\p A$ and $\Gv_z$ on $\p D$ is given by
\begin{align}
    \Gv_\Gz = -\frac{T'(z)}{|T'(z)|} \Gv_z = \frac{\overline{z}-\Gb/2}{z- \Gb/2} \Gv_z,
\end{align}
where the minus sign is taken due to the fact that $T$ maps from the inside $D$ to the outside $A$. Therefore, we have
\begin{equation} \label{eq_normal}
\p_{\Gv_{\Gz}} = 2\Re \left( \Gv_\Gz  \frac{\p}{\p \Gz} \right) = -\frac{\Gb}{|\Gz - 1|^2}2\Re \left( \Gv_z  \frac{\p}{\p z} \right) = -\frac{\Gb}{|\Gz - 1|^2} \p_{\Gv_z}.
\end{equation}
Thus, the second line in \eqref{LCC} is translated to
$$
U|_- + \frac{\Gg |\zeta-1|^2}{\Gb} \p_\Gv U|_- = \text{const.} \quad \text{on } \p A. 
$$
One can also see that the third line in \eqref{LCC} is translated to
$$
    \int_{|\zeta|=\Gr^{\pm 1}} \p_\Gv U|_-\, \df \sigma=0.
$$
So, $U$ is the solution to the following problem which we denote by \eqref{LCCA} (A for annulus):
\begin{equation}\leqnomode
\label{LCCA}\tag{LC-C-A}
\begin{cases}
    \GD U=0 & \text{in } A \setminus \{ 1 \}, \\
    \ds U|_- + \frac{\Gg |\zeta-1|^2}{\Gb} \p_\Gv U|_- = \text{const.} & \text{on } \p A, \\
    \ds \int_{|\zeta|=\Gr^{\pm 1}} \p_\Gv U|_- \, \df \sigma=0,  \\
    \ds U (\zeta)-\frac{\Gb}{2}-\Re\frac{\Gb}{\zeta-1}=O(|\zeta-1|) & \text{as } \Gz\to 1,
\end{cases}
\end{equation}
where the last condition holds since we deal with the case when $h(z)=\Re z$. Here $[\cdot]|_-$ is the boundary value from the interior of $A$ and the normal vector $\nv$ is outward from $A$.

\subsection{Construction of the solution}

Let $U$ be the solution to \eqref{LCCA}. Then, $U(\Gz) - \Re \frac{\Gb}{\Gz - 1} - \frac{\Gb}{2}$ is a real-valued harmonic function in $A$, and hence $U$ takes the form
$$
    U(\Gz) = \Re \frac{\Gb}{\Gz - 1} + \frac{\Gb}{2} + \Re \left(f(\Gz) + g(\Gz)\right)  + C \ln |\Gz|
$$
for some constant $C$, where $f$ is a holomorphic function in $|\Gz| < \Gr^{-1}$ and $g$ a holomorphic function in $|\Gz| > \Gr$ such that $g(\Gz) \to 0$ as $|\Gz| \to \infty$ (see, for example, \cite{ABR01}).

Because of the third condition in \eqref{LCCA}, we have $C=0$. Note that $D$ is symmetric with respect to the imaginary axis. Since $h(z)$ is anti-symmetric with respect to the imaginary axis, so is the solution $u$ by the uniqueness of the solution. Thus $U$ is anti-symmetric with respect to the inversion over $|\Gz|=1$, namely,
\begin{align} \label{symmetry_U}
U(\Gz) = - U(\overline{\Gz}^{-1}).
\end{align}
So, $U$ is of the form
\begin{equation}\label{eq_sol_disks}
    U(\Gz)=\frac{\Gb}{2}+\Re \left(\frac{\Gb}{\Gz-1}+F (\Gz)-F (\Gz^{-1})\right)
\end{equation}
for some holomorphic function $F (\Gz)$ of the form
\begin{equation}
    \label{eq_F_exp}
    F (\Gz)=\sum_{n=1}^\infty c(n)\Gz^n \quad (c: \Nbb \to \Rbb).
\end{equation}

It will turn out that the radius of convergence for the above power series is at least $\Gr^{-2}$, which is greater than $\Gr^{-1}$ (Theorem \ref{theo_convergence}). Thus the function $U$ in \eqref{eq_sol_disks} satisfies \eqref{LCCA} except the second line.
In what follows, we determine the coefficients $c(n)$ so that $U$ satisfies the second line.

We obtain the cosine-Fourier expansions on $\p A$:
\[
    U (\Gr\e^{\iu\theta})=-\frac{\Gb}{2}+\sum_{n=1}^\infty (-\Gb \Gr^n-(\Gr^{-n}-\Gr^n)c(n) )\cos (n\theta)
\]
and
\[
    U (\Gr^{-1}\e^{\iu\theta})=\frac{\Gb}{2}+\sum_{n=1}^\infty (\Gb \Gr^n+(\Gr^{-n}-\Gr^n)c(n))\cos (n\theta).
\]
Furthermore, since $\p_\nv=-\p_{|\Gz|}$ on $\p B_{\Gr}(0)$ and $\p_\nv=\p_{|\Gz|}$ on $\p B_{\Gr^{-1}}(0)$, we obtain
\begin{align*}
    \p_\Gv U|_-(\Gr\e^{\iu\theta})
    &=\sum_{n=1}^\infty n (\Gb \Gr^{n-1}-(\Gr^{n-1}+\Gr^{-n-1})c(n))\cos (n\theta) \\
    &=\Gr^{-1}\sum_{n=1}^\infty a (n) \cos (n\theta)
\end{align*}
and
\begin{align*}
    \p_\Gv U|_- (\Gr^{-1}\e^{\iu\theta})
    &=\sum_{n=1}^\infty n (-\Gb \Gr^{n+1}+(\Gr^{-n+1}+\Gr^{n+1})c(n))\cos (n\theta) \\
    &=-\Gr\sum_{n=1}^\infty a (n) \cos (n\theta),
\end{align*}
where we set
\begin{equation}\label{eq_ab}
\begin{aligned}
    a (n)&:=n (\Gb \Gr^n -(\Gr^{-n}+\Gr^n)c(n)).
\end{aligned}
\end{equation}

Let $\lap_\mathrm{D}$ be the the discrete Laplacian on $\Nbb$ with the Dirichlet boundary condition, namely, for a complex sequence $\xi: \Nbb\to \Cbb$,
\[
    (\lap_\mathrm{D}\xi) (n):=
    \begin{cases}
        \xi (n+1)+\xi (n-1)-2\xi (n) & \text{if } n\geq 2, \\
        \xi (2)-2\xi (1) & \text{if } n=1.
    \end{cases}
\]
Set
\[
    \mu:=\Gr+\Gr^{-1}-2.
\]
Then, we have
\begin{align*}
    |\Gr\e^{\iu\theta}-1|^2 \p_\Gv U|_- (\Gr\e^{\iu\theta}) =-a (1)
    +\sum_{n=1}^\infty (-\lap_\mathrm{D}+\mu I)a (n) \cos (n\theta).
\end{align*}
and
\begin{align*}
    |\Gr^{-1}\e^{\iu\theta}-1|^2 \p_\Gv U|_- (\Gr^{-1}\e^{\iu\theta})
    =a (1)-\sum_{n=1}^\infty (-\lap_\mathrm{D}+\mu I)a (n) \cos (n\theta).
\end{align*}
We substitute all the above identities to the second condition in \eqref{LCCA} to obtain
\begin{equation}
    \label{eq_imperfect_annulus_recursive_c0}
    \begin{dcases}
        \Gl_i +\frac{\Gb}{2}=-\frac{\Gg}{\Gb}a (1), \\
        \Gl_e -\frac{\Gb}{2}=\frac{\Gg}{\Gb}a (1)
    \end{dcases}
\end{equation}
and
\begin{equation}
    \label{eq_imperfect_annulus_recursive}
        \Gb \Gr^n+(\Gr^{-n}-\Gr^n)c(n)=\frac{\Gg}{\Gb}(-\lap_\mathrm{D}+\mu I)a (n)
\end{equation}
for $n\geq 1$, where $\Gl_i, \Gl_e$ are constants on ${|\Gz|=\Gr}$ and ${|\Gz|=\Gr^{-1}}$, respectively, to be determined.

The relation \eqref{eq_ab} yields
\begin{equation}\label{eq_cn_ab}
\begin{aligned}
    c(n)&=-\frac{\Gr^na (n)}{n(1+\Gr^{2n})}+\frac{\Gb \Gr^{2n}}{1+\Gr^{2n}},
\end{aligned}
\end{equation}
and hence \eqref{eq_imperfect_annulus_recursive} becomes
\begin{equation}
    \label{eq_imperfect_annulus_recursive_2}
    \frac{2\Gr^n}{1+\Gr^{2n}}-\frac{1-\Gr^{2n}}{\Gb n(1+\Gr^{2n})}a (n)=\frac{\Gg}{\Gb^2}(-\lap_\mathrm{D}+\mu I)a (n)
\end{equation}
for $n\geq 1$. We define the multiplication operator
\[
    (V \xi) (n):=\frac{1- \Gr^{2n}}{\Gb n(1+ \Gr^{2n})}\xi (n)
\]
acting on a complex sequence $\xi: \Nbb\to \Cbb$. We set
\begin{align*}
    f (n):=\frac{2\Gr^n}{1+\Gr^{2n}} \quad (n\in \Nbb).
\end{align*}
Then, \eqref{eq_imperfect_annulus_recursive_2} reads
\begin{equation}
    \label{eq_imperfect_annulus_recursive_3}
    \frac{\Gg}{\Gb^2}(-\lap_\mathrm{D}+\mu I)a+V a=f.
\end{equation}

We find the solution to \eqref{eq_imperfect_annulus_recursive_3} in the Hilbert space $l^2 (\Nbb, \Rbb)$ and investigate the asymptotic behavior as $n\to\infty$. To do so, we define the discrete Schr\"odinger operator
\beq\label{Hdef}
    H:=\frac{\Gg}{\Gb^2}(-\lap_\mathrm{D}+\mu I)+V.
\eeq

\begin{theo}\label{theo_inv_weighted}
    For $\Kbb=\Rbb, \Cbb$ and any $\Gg>0$, the operator $H$ is an invertible bounded linear operator on $l^2 (\Nbb, \Kbb)$.
\end{theo}

\begin{proof}
    It suffices to prove the case when $\Kbb=\Cbb$. It is easy to see the boundedness and the positivity of $-\lap_\mathrm{D}$ on $l^2 (\Nbb, \Cbb)$ by the sine-Fourier transform. Since $-\lap_\mathrm{D}$ and $V$ are bounded positive operators on $l^2 (\Nbb, \Cbb)$, we obtain the boundedness of $H$ on $l^2 (\Nbb, \Cbb)$ and the inequality
    \begin{align*}
        \jbk{H \xi, \xi}_{l^2}
        =\frac{\gamma}{\beta^2} (\jbk{-\lap_\mathrm{D}\xi, \xi}_{l^2}+\mu \|\xi\|_{l^2}^2)+\jbk{V\xi, \xi}_{l^2}\geq \frac{\gamma\mu}{\beta^2}\|\xi\|_{l^2}^2
    \end{align*}
    for any $\xi\in l^2 (\Nbb, \Cbb)$. Thus the spectrum of $H$ is included in the closed set $[\Gg\mu/\Gb^2, \infty)$, which does not contain $0$. Hence $H$ is invertible.
\end{proof}

To look into the decay rate of $a (n)$ ($a$ is the solution to \eqref{eq_imperfect_annulus_recursive_3}) as $n\to\infty$, we investigate the behavior of super- and sub-solution for the operator $H$:
\begin{theo}
    \label{theo_comparison}
    If $\xi\in l^\infty (\Nbb, \Rbb)$ satisfies $H\xi (n)\geq 0$ (resp. $H\xi (n)\leq 0$) for all $n\in \Nbb$, then $\xi (n)\geq 0$ (resp. $\xi (n)\leq 0$) for all $n\in \Nbb$.
\end{theo}

For the proof of Theorem \ref{theo_comparison}, we prepare the following lemma, which is a discrete analogue of the maximum principle.

\begin{lemm}
    \label{lemm_discrete_lap}
    Let $I=\Nbb \cap [N, \infty)$ for some $N\in \Nbb\cup \{0\}$. If $\xi: I\to \Rbb$ satisfies $-\lap_\mathrm{D}\xi\geq 0$ (resp. $-\lap_\mathrm{D}\xi \leq 0$) on $I\setminus \{N\}$ and is bounded from below (resp. from above), then the following alternative holds.
    \begin{enumerate}[label={\rm (\alph*)}]
        \item $\xi$ is constant on $I$.
        \item $\xi$ is monotonically increasing with $\xi (N)<\xi (N+1)$ (resp. monotonically decreasing with $\xi (N)>\xi (N+1)$), where we set $\xi (0)=0$ if $N=0$.
    \end{enumerate}
    In particular, $\xi$ is monotonically increasing (resp. decreasing) in $I$.
\end{lemm}

\begin{proof}
    We only deal with the case when $-\lap_\mathrm{D}\xi\geq 0$ and $\xi$ is bounded from below. Since $-\lap_\mathrm{D} \xi\geq 0$, we have
    \begin{equation}\label{eq_mp_lap_grad_monot}
        \xi(n+1)-\xi (n)\leq \xi (n)-\xi (n-1)
    \end{equation}
    for all $n\in I\setminus \{N\}$. For any $m \in I\setminus \{N\}$, we have
    \begin{align*}
        \xi (n)&=\xi (m-1)+\sum_{k=m}^n (\xi (k)-\xi (k-1))
        \\
        &\leq \xi (m-1)+(n-m+1)(\xi (m)-\xi (m-1))
    \end{align*}
    for all $n>m$ by \eqref{eq_mp_lap_grad_monot}. If $\xi (m)-\xi (m-1)<0$ for some $m\in I\setminus \{N\}$, then we infer from the above inequality that
    \[
        \inf_{n>m} \xi (n) = -\infty
    \]
    which contradicts to the boundedness from below of $\xi$. Thus $\xi (n)-\xi (n-1)\geq 0$ for all $n\in I\setminus \{N\}$.

    Assume that $\xi $ is not constant. Then there exists $m\in I\setminus \{N\}$ such that $\xi (m)>\xi (m-1)$. Since $\xi (m)-\xi (m-1)\leq \xi (N+1)-\xi (N)$ by \eqref{eq_mp_lap_grad_monot}, we have $\xi (N+1)>\xi (N)$.
\end{proof}

\begin{proof}[Proof of Theorem \ref{theo_comparison}]
    We only deal with the case when $H\xi (n)\geq 0$ for all $n\in \Nbb$. We prove the assertion by an argument analogous to the continuous case as in \cite[6.4.1]{Evans10}. 
    
    We set $\xi (0)=0$ for convenience and put
    \[
        I^\circ:=\{ n\in \Nbb \mid \xi (n)<0\}.
    \]
    Then, for any $n\in I^\circ$, we have
    \beq\label{positive}
        -\frac{\Gg}{\Gb^2}\lap_\mathrm{D}\xi (n)=H\xi (n)-\frac{\Gg\mu}{\Gb^2}\xi (n)-V\xi (n)\geq 0.
    \eeq
    In particular, the inequality \eqref{eq_mp_lap_grad_monot} holds for all $n\in I^\circ$.

    We claim that
    \beq\label{ifn}
    \text{if $n\in I^\circ$, then $n+1\in I^\circ$}.
    \eeq
    In fact, if there is $n\in I^\circ$ such that $n+1\not\in I^\circ$, then $\xi (n)<\xi (n+1)$. Thus by \eqref{eq_mp_lap_grad_monot}, we obtain $\xi (n)-\xi (n-1)>0$ and thus $n-1\in I^\circ$. Again, by \eqref{eq_mp_lap_grad_monot} with $n$ replaced to $n-1$, we obtain $\xi (n-1)-\xi (n-2)>0$ and hence $n-2\in I^\circ$. Repeating this procedure, we have $\xi (2)-\xi (1)>0$. Thus, we have $\{ 1, 2, \cdots, n \} \subset I^\circ$. However, since
    \[
        -\lap_\mathrm{D}\xi (1)=-\xi (2)+2\xi (1)<\xi (1)<0,
    \]
    we arrive at a contradiction to \eqref{positive}.

    We now prove that $I^\circ= \emptyset$. Suppose on the contrary that $I^\circ \not= \emptyset$. We set $N:=\min I^\circ-1$. By \eqref{ifn}, we have $I^\circ=\Nbb\cap [N+1, \infty)$. Since $\xi$ is bounded by assumption, we infer from Lemma \ref{lemm_discrete_lap} for $I:=I^\circ \cup \{N\}$ that $\xi$ is monotonically  increasing in $I$. Thus we have $\xi (N)\leq \xi (N+1)<0$. This is not possible since $N=\min I^\circ-1$ if $N >0$ and $\xi (0)=0$ if $N=0$. Therefore the set $I^\circ$ must be empty.
\end{proof}

By Theorem \ref{theo_comparison}, we derive the following estimates of $a (n)$.

\begin{theo}
    \label{theo_agn_estimate}
    With the constants
    \begin{equation}
        \label{eq_coeffs}
        B_1:=2\left(1+\frac{\Gg}{r}\right)^{-1}\quad\text{and} \quad
        B_2:=2\left(1+\frac{\Gg}{r}-\Gr^2 \left(1-\frac{\Gg}{r}\right)\right)^{-1},
    \end{equation}
    it holds that
    \begin{numcases}{}
        B_1 \Gb n\Gr^n \leq a (n) \leq B_2 \Gb n  \Gr^n & if $0<\Gg<r$, \label{case_agn_estimate_small_g} \\
        a (n)=B_1 \Gb n \Gr^n & if $\Gg=r$, \label{case_agn_estimate_eq_g} \\
        B_2 \Gb n \Gr^n \leq a (n) \leq B_1 \Gb n \Gr^n & if $\Gg>r$ \label{case_agn_estimate_large_g}
        \end{numcases}
    for all $n\in \Nbb$.
\end{theo}

\begin{proof}
    Let $g$ be the sequence defined by
    \[
    g (n):=n\Gr^n, \quad n \in \Nbb.
    \]
    Then, straight-forward computations, using \eqref{eq_imperfect_annulus_recursive_3}, \eqref{Hdef}, and the formula
    \begin{equation}\label{eq_rho_beta_rel}
        \frac{\Gr^{-1}-\Gr}{\Gb}=\frac{1}{r},
    \end{equation}
    yield
    \begin{equation}\label{eq_comparison_app}
        H(a-C\Gb g)(n)
        =\frac{\Gr^n}{1+\Gr^{2n}} \left(2-C\left(1+\frac{\Gg}{r}-\left(1-\frac{\Gg}{r}\right)\Gr^{2n}\right)\right), \quad n\in \Nbb,
    \end{equation}
    for any constant $C$.

    If $\Gg=r$, then we take $C=B_1=B_2$ to have $H(a-B_1\Gb g)=0$. Then \eqref{case_agn_estimate_eq_g} follows from
    Theorem \ref{theo_inv_weighted}.

    If $0<\Gg<r$, then $H(a-B_1 \Gb g)\geq 0$ and thus $a\geq B_1 \Gb g$ on $\Nbb$ by Theorem \ref{theo_comparison}. Furthermore, if we set $C=B_2\geq 0$, then we have
    \begin{align*}
        H(a-B_2 \Gb g)(n)
        \leq \frac{\Gr^n}{1+\Gr^{2n}} \left(2-C\left(1+\frac{\Gg}{r}-\left(1-\frac{\Gg}{r}\right)\Gr^2\right)\right)= 0
    \end{align*}
    for all $n\in \Nbb$ from \eqref{eq_comparison_app} (since $0<\Gr<1$). Hence we obtain $a\leq B_2 \Gb g$ by Theorem \ref{theo_comparison}.

    If $\Gg>r$, then $H(a-B_1 \Gb g)\leq 0$ and thus $a\leq B_1 \Gb g$ on $\Nbb$ by Theorem \ref{theo_comparison}. Similarly to the $\Gg<r$ case, one can show that $H(a-B_2 \Gb g)\geq 0$ on $\Nbb$. Hence we obtain $a\geq B_2 \Gb g$ by Theorem \ref{theo_comparison}.
\end{proof}

\begin{rema*}
    If $\Gg=r$, then the formula \eqref{case_agn_estimate_eq_g} is exact. If we substitute it to \eqref{eq_F_exp} (and \eqref{eq_cn_ab}), we obtain $U=\Gb/2+\Re \Gb/(\Gz-1)$. Thus the solution to the problem \eqref{LCC} when inclusions are disks given by \eqref{D1D2} and $h(z)= \Re z$ is $u (z)=(U\circ T) (z)=\Re z$. Note that without the inclusions, the solution is nothing but $h(z)=\Re z$. So, the gradient of the solution is not perturbed by the insertion of inclusions. Such an inclusion is called a neutral inclusion to the given field $-\nabla h$. It is proved in \cite{Torquato-Rintoul95} (see also \cite{KL19}) that a single disk of radius $r$ (with the imperfect bonding parameter $\Gg=r$) is neutral to constant fields. So insertion of two such inclusions does not perturb the uniform fields. What we see here is that two disks of the same radius $r$ of the configuration \eqref{D1D2} is neutral to the field $(1,0)$. It is so independently of the distance between two inclusions. In particular, the gradient of the solution does not change depending on the distance. The neutrality condition for the core-shell structure was discovered by Hashin \cite{Hashin62}: if $k_c=\infty$ in Figure \ref{fg_core_shell}, then the condition is given by
    $$
    k_s= \frac{r^2-(r-t)^2}{r^2+(r-t)^2} 
    $$
    ($r$ is the radius of the outer circle). The above neutrality conditions for the core-shell structure and the imperfect bonding interface are consistent in the sense that 
    \[
        \gamma=\left(\lim_{t\to 0}\frac{k_s}{t}\right)^{-1}=r. 
    \]
\end{rema*}

Let
    \begin{equation}
        \label{eq_coeff_big}
        C_D:= \frac{1}{1-\Gr+\Gg/r}.
    \end{equation}
One can easily see that $\max \{ B_1, B_2\}\leq 2C_D$. For example, we have
    \[
        B_2 =2\left((1+\Gr)(1-\Gr)+\frac{\Gg}{r}(1+\Gr^2)\right)^{-1}
        \leq 2\left(1-\Gr+\frac{\Gg}{r}\right)^{-1}.
    \]
So, as an immediate consequence of Theorem \ref{theo_agn_estimate}, we obtain the following corollary.
\begin{coro}
    \label{coro_agn_estimate_rough}
    It holds that
    \begin{equation}
        \label{eq_agn_estimate_rough}
        0\leq a (n)\leq 2C_D \Gb n \Gr^n
    \end{equation}
    for any $r, \Ge, \Gg>0$ and $n\in \Nbb$.
\end{coro}

Now we prove the convergence of the series \eqref{eq_F_exp}.

\begin{theo}
    \label{theo_convergence}
    For any $\Gg>0$, the convergence radius of the series \eqref{eq_F_exp} is at least $\Gr^{-2}$. In particular, the function $U (\Gz)$ is smooth on $\overline{A}$.
\end{theo}

\begin{proof}
    By Theorem \ref{theo_agn_estimate} and \eqref{eq_cn_ab}, we have
    \begin{equation}\label{eq_cn_estimate}
        |c(n)|\leq \frac{\Gb\Gr^{2n}}{1+\Gr^{2n}}\max\{ |1-B_1|, |1-B_2|\}
        \leq \frac{2C_D \Gb \Gr^{2n}}{1+\Gr^{2n}}\left|1-\frac{\Gg}{r}\right| .
    \end{equation}
 We then obtain
 $$
 \limsup_{n\to\infty}|c(n)|^{1/n}\leq \Gr^2,
 $$
 and the desired conclusion by the Cauchy-Hadamard theorem.
\end{proof}

So far we constructed a solution $U (\Gz)$ to the transformed imperfect bonding problem \eqref{LCCA} in the form \eqref{eq_sol_disks}. In the next subsection, we derive the gradient estimate of the solution.

\subsection{Gradient estimate of the solution}

Recall from \eqref{eq_sol_disks} that $U$ is of the form
$$
    U(\Gz)=\frac{\Gb}{2}+\Re \left(\frac{\Gb}{\Gz-1}+F (\Gz)-F (\Gz^{-1})\right),
$$
where $F (\Gz)$ is a holomorphic function of the form
$$
    F (\Gz)=\sum_{n=1}^\infty c(n)\Gz^n.
$$

We apply the chain rule for $u=U \circ T$ with $\Gz=T (z)$ to obtain
\begin{align}
    |\nabla u (z)|&= \left|\p_z u (z)\right| \nonumber \\
    &=\left| 1-\frac{(\Gz-1)^2}{\Gb}\left(\p_\Gz F (\Gz)+\frac{1}{\Gz^2}\p_\Gz F (\Gz^{-1})\right)\right| \nonumber \\
    &\leq 1+\frac{1}{\Gb} (|(\Gz-1)^2\p_\Gz F (\Gz)|+|(\Gz^{-1}-1)^2\p_\Gz F (\Gz^{-1})|) \nonumber \\
    &\leq 1+\frac{2}{\Gb}\| (\Gz-1)^2 \p_\Gz F (\Gz)\|_{L^\infty (A)}, \label{eq_grad_u}
\end{align}
where the last equality holds since the annulus $A$ is invariant under the inversion $\Gz\mapsto \Gz^{-1}$. Thus it suffices to estimate $|(\Gz-1)^2 \p_\Gz F (\Gz)|$ on $A$.

By Theorem \ref{theo_convergence}, we have
\[
    \p_\Gz F (\Gz)
    =\sum_{n=0}^\infty (n+1)c (n+1)\Gz^n=:\sum_{n=0}^\infty d (n+1)\Gz^n,
\]
there the definition of $d(n)$ is apparent. We use the easy formula $\Gz^n=(\Gz^{n+1}-\Gz^n)/(\Gz-1)$ and apply the summation by parts to obtain
\begin{align*}
    \p_\Gz F (\Gz)
    =-\frac{c (1)}{\Gz-1}-\frac{1}{\Gz-1}\sum_{n=1}^\infty (d (n+1)-d (n))\Gz^n.
\end{align*}
We repeat this procedure once again to obtain
\[
    \p_\Gz F (\Gz)
        =-\frac{c (1)}{\Gz-1}+\frac{(2c (2)-c (1))\Gz}{(\Gz-1)^2}
        +\frac{1}{(\Gz-1)^2}\sum_{n=2}^\infty (\lap_\mathrm{D}d)(n)\Gz^n.
\]
We then arrive at
\begin{equation}
    \label{eq_grad_abel}
        (\Gz-1)^2 \p_\Gz F (\Gz)
        =-c (1)(\Gz-1)+(2c (2)-c (1))\Gz
        +\sum_{n=2}^\infty (\lap_\mathrm{D}d)(n)\Gz^n.
\end{equation}

In what follows, we estimate each term in \eqref{eq_grad_abel}.

\begin{lemm}
    \label{lemm_estimate_beginning}
    We have the estimate
    \begin{equation}
        \label{eq_estimate_beginning}
        \begin{aligned}
            \left|-c (1)(\Gz-1)+(2c (2)-c (1))\Gz\right|\leq 10\Gb C_D \left|1-\frac{\Gg}{r}\right|
        \end{aligned}
    \end{equation}
    for all $r, \Ge, \Gg>0$ and $\Gz\in A$.
\end{lemm}

\begin{proof}
We have
    \[
        |-c (1)(\Gz-1)+(2c (2)-c (1))\Gz|
        \leq |c (1)|(|\Gz|+1)+(|2c (2)|+|c (1)|)|\Gz|.
    \]
So, \eqref{eq_estimate_beginning} follows from the estimate \eqref{eq_cn_estimate} for $|c(n)|$.
\end{proof}

Next, we estimate $\sum_{n=2}^\infty (\lap_\mathrm{D} d)(n)\Gz^n$. To do so, let $\nabla_\mathrm{D}$ be the forward difference of $\xi: \Nbb\to \Cbb$, namely,
\[
    (\nabla_\mathrm{D} \xi)(n):=\xi (n+1)-\xi (n), \quad n \in \Nbb.
\]
Then, the following formula holds for $\xi, \eta: \Nbb\to \Cbb$ with the convention $\xi (0)=\eta (0)=0$:
\begin{align*}
    \lap_\mathrm{D}(\xi \eta)(n)=&\,(\lap_\mathrm{D}\xi)(n)\eta (n)+\xi (n)\lap_\mathrm{D}\eta (n) \\
    &+(\nabla_\mathrm{D} \xi)(n)(\nabla_\mathrm{D} \eta)(n)+(\nabla_\mathrm{D} \xi)(n-1)(\nabla_\mathrm{D} \eta)(n-1).
\end{align*}

By \eqref{eq_cn_ab}, we have
\[
    d (n)=\frac{\Gb n \Gr^{2n}}{1+\Gr^{2n}}-\frac{\Gr^n a (n)}{1+\Gr^{2n}}, \quad n\in \Nbb.
\]
Furthermore, we have
\begin{align*}
    &\lap_\mathrm{D}\left(\frac{\Gr^n}{1+\Gr^{2n}}a\right)(n) \\
    &=a (n)\left(\lap_\mathrm{D}\frac{\Gr^n}{1+\Gr^{2n}}\right)(n)+\frac{\Gr^n}{1+\Gr^{2n}}\lap_\mathrm{D}a (n) \\
    &\quad +\left(\nabla_\mathrm{D} \frac{\Gr^n}{1+\Gr^{2n}}\right)(n)\nabla_\mathrm{D} a (n)+\left(\nabla_\mathrm{D} \frac{\Gr^n}{1+\Gr^{2n}}\right)(n-1)\nabla_\mathrm{D} a (n-1)
\end{align*}
for $n\geq 2$. Thus we have
\begin{equation}
    \label{eq_dn_decompose}
    \begin{aligned}
        &\sum_{n=2}^\infty (\lap_\mathrm{D} d)(n)\Gz^n \\
        &=\Gb \sum_{n=2}^\infty \left(\lap_\mathrm{D} \frac{n \Gr^{2n}}{1+\Gr^{2n}}\right)(n)\Gz^n
        -\sum_{n=2}^\infty a (n)\left(\lap_\mathrm{D}\frac{\Gr^n}{1+\Gr^{2n}}\right)(n)\Gz^n \\
        &\quad -\sum_{n=2}^\infty \frac{\Gr^n}{1+\Gr^{2n}}\lap_\mathrm{D}a (n) \Gz^n
        -\sum_{n=2}^\infty \left(\nabla_\mathrm{D} \frac{\Gr^n}{1+\Gr^{2n}}\right)(n)\nabla_\mathrm{D} a (n)\Gz^n \\
        &\quad -\sum_{n=2}^\infty \left(\nabla_\mathrm{D} \frac{\Gr^n}{1+\Gr^{2n}}\right)(n-1)\nabla_\mathrm{D} a (n-1) \Gz^n \\
        &=:\Gb S_0-S_1-S_2-S_3-S_4.
    \end{aligned}
\end{equation}

\begin{lemm}
    \label{lemm_S0}
    It holds that
    \begin{equation}\label{eq_S0}
        |S_0|\leq 8
    \end{equation}
    for all $r, \Ge, \Gg>0$ and $\Gz\in A$.
\end{lemm}

\begin{proof}
    Since
    \begin{align*}
        &\lap_\mathrm{D}\left(\frac{n\Gr^{2n}}{1+\Gr^{2n}}\right)(n) \\
        &=\frac{n\Gr^{2n}(\Gr^{-1}-\Gr)^2 (1-\Gr^{2n})}{(1+\Gr^{2n+2})(1+\Gr^{2n-2})(1+\Gr^{2n})}
        -\frac{\Gr^{2n}(\Gr^{-2}-\Gr^2)}{(1+\Gr^{2n+2})(1+\Gr^{2n-2})}
    \end{align*}
    for $n\geq 2$, we have
    \begin{align*}
        |S_0|&\leq (\Gr^{-1}-\Gr)^2 \sum_{n=2}^\infty n\Gr^{2n}|\Gz|^n +(\Gr^{-2}-\Gr^2)\sum_{n=2}^\infty \Gr^{2n}|\Gz|^n  \\
        &\leq (\Gr^{-1}-\Gr)^2 \sum_{n=2}^\infty n\Gr^n +(\Gr^{-2}-\Gr^2)\sum_{n=2}^\infty \Gr^n \\
        &=(1+\Gr)(3+\Gr)\leq 8,
    \end{align*}
as desired. \end{proof}

\begin{lemm}
    \label{lemm_S1}
    It holds that
    \begin{equation}
        \label{eq_S1}
        |S_1|\leq 6\Gb C_D
    \end{equation}
    for all $r, \Ge, \Gg>0$ and $\Gz\in A$.
\end{lemm}

\begin{proof}
    From the identity
    \begin{align*}
        &\left(\lap_\mathrm{D}\frac{\Gr^n}{1+\Gr^{2n}}\right)(n) \\
        &=\frac{(1-\Gr)^2(1+\Gr^{2n})\Gr^{n-1}}{(1+\Gr^{2n+2})(1+\Gr^{2n-2})}-\frac{(1-\Gr)^2 \Gr^{2n-2}}{(1+\Gr^{2n+2})(1+\Gr^{2n})(1+\Gr^{2n-2})},
    \end{align*}
    we have
    \[
        \left|\left(\lap_\mathrm{D}\frac{\Gr^n}{1+\Gr^{2n}}\right)(n)\right|
        \leq 3(1-\Gr)^2\Gr^{n-1}
    \]
    for $n\geq 2$. It then follows from Corollary \ref{coro_agn_estimate_rough} that
    \begin{align*}
        &\left|\sum_{n=2}^\infty \left(\lap_\mathrm{D}\frac{\Gr^n}{1+\Gr^{2n}}\right)(n)a (n) \Gz^n\right| \\
        &\leq \sum_{n=2}^\infty \left|\left(\lap_\mathrm{D}\frac{\Gr^n}{1+\Gr^{2n}}\right)(n)\right||a (n)| |\Gz|^n \\
        &\leq 6\Gb C_D (1-\Gr)^2 \sum_{n=2}^\infty n\Gr^{n-1}
        =6\Gb C_D \Gr (2-\Gr)\leq 6\Gb C_D,
    \end{align*}
    as desired.
\end{proof}

To estimate $S_2$, $S_3$ and $S_4$, we consider $\nabla_\mathrm{D} a$ and $\lap_\mathrm{D}a$.

\begin{theo}
    \label{theo_diff_estimate}
    It holds that
    \begin{equation}
        \label{eq_lap_estimate}
        |\lap_\mathrm{D}a (n)|\leq 2(1-\Gr)^2 (\beta n+r)\Gr^{n-1}+\frac{6\beta^2 \Gr^n}{\gamma}
    \end{equation}
    and
    \begin{equation}
        \label{eq_diff_estimate}
        |\nabla_\mathrm{D} a (n)|
        \leq \frac{6r\beta}{\gamma}((1-\Gr)n+4)\Gr^{n-1}.
    \end{equation}
\end{theo}

The above theorem is proved by using the following lemma.

\begin{lemm}
    \label{lemm_agn_estimate_2}
    It holds that
    \begin{equation}
        \label{eq_agn_estimate_rough_2}
        0\leq a (n)\leq 2(\beta n+r)\Gr^n
    \end{equation}
    for any $r, \Ge, \gamma>0$ and $n\in \Nbb$.
\end{lemm}

\begin{proof}
    It suffices to prove the inequality
    \begin{equation}
        \label{case_agn_estimate_small_g_2}
        a (n)\leq (B_1 \beta n+R)\Gr^n
    \end{equation}
    for all $r, \Ge, \gamma>0$ with $0<\gamma<r$ and $n\in \Nbb$ where $B_1$ is defined by \eqref{eq_coeffs} and
    \[
        R:=2r\Gr \left(1-\frac{\gamma}{r}\right) \left(1+\frac{\gamma}{r}\right)^{-1}
    \]
    by virtue of Theorem \ref{theo_agn_estimate}, $B_1\leq 2$ and $R\leq 2r$.

    We define the sequence $\widetilde{g} (n):=(B_1 \beta n+R)\Gr^n$. Then we have  from \eqref{eq_imperfect_annulus_recursive_3} and the formula \eqref{eq_rho_beta_rel}
    \begin{equation}\label{eq_comparison_app_2}
        H(a-\widetilde{g})(n)
        =\frac{\Gr^{3n}}{1+\Gr^{2n}} \left(B_1\left(1-\frac{\gamma}{r}\right)-\frac{R (\Gr^{-2n}-1)}{\beta n}\right)-\frac{\gamma R}{\beta^2}\delta_{n1}
    \end{equation}
    for all $n\in \Nbb$, where $\delta_{mn}$
    is the Kronecker delta. Since $x\mapsto (\Gr^{-2x}-1)/x$ is a monotonically increasing function on $(0, \infty)$, we have
    \begin{align*}
        H(a-\widetilde{g})(n)
        \leq \frac{\Gr^{3n}}{1+\Gr^{2n}} \left(B_1\left(1-\frac{\gamma}{r}\right)-\frac{R (\Gr^{-2}-1)}{\beta}\right)= 0
    \end{align*}
    for all $n\in \Nbb$ from \eqref{eq_rho_beta_rel}, \eqref{eq_comparison_app_2} and $0<\Gr<1$. Hence we obtain $a\leq \widetilde{g}$ from Theorem \ref{theo_comparison}.
\end{proof}

\begin{proof}[Proof of Theorem \ref{theo_diff_estimate}]
    We begin with noting the estimate
    \begin{equation}
        \label{eq_Vaf}
        Va (n)+f(n)\leq 6\Gr^n
    \end{equation}
    for all $n\in \Nbb$. In fact, by Lemma \ref{lemm_agn_estimate_2} and \eqref{eq_rho_beta_rel}, we have
    \begin{align*}
        Va (n)+f(n)&\leq \frac{2(1-\Gr^{2n})}{\beta n (1+\Gr^{2n})}(\beta n +r)\Gr^n+2\Gr^n
        \leq 4\Gr^n+\frac{2r(1-\Gr^{2n})}{\beta n}\Gr^n \\
        &=4\Gr^n+\frac{2r(1-\Gr^2)(1+\Gr^2+\Gr^4+\cdots + \Gr^{2n-2})}{\beta n}\Gr^n \\
        &\leq 4\Gr^n+\frac{2r(1-\Gr^2)}{\beta}\Gr^n=4\Gr^n+2\Gr^{n+1}\leq 6\Gr^n.
    \end{align*}

    We apply Corollary \ref{coro_agn_estimate_rough} and Lemma \ref{lemm_agn_estimate_2} to \eqref{eq_imperfect_annulus_recursive_3} and obtain
    \begin{align*}
        |\lap_\mathrm{D}a (n)|
        &\leq \mu a (n)+\frac{\beta^2}{\gamma}(Va (n)+f(n)) \\
        &\leq 2(1-\Gr)^2 (\beta n+r)\Gr^{n-1}+\frac{6\beta^2 \Gr^n}{\gamma},
    \end{align*}
which is the desired estimate \eqref{eq_lap_estimate}.

We now prove \eqref{eq_diff_estimate}.
    In what follows, we set $a (0):=0$. We rewrite \eqref{eq_imperfect_annulus_recursive_3} as follows:
    \[
        a (n+1)-\Gr a (n)
        =\Gr^{-1}(a (n)-\Gr a (n-1))+\frac{\Gb^2}{\Gg}(Va (n)-f(n)), \quad n\in \Nbb.
    \]
    We then have
    \[
        a (n)-\Gr a (n-1)=\Gr^{-n+1}a (1)+\frac{\Gb^2}{\Gg}\sum_{k=1}^{n-1}\Gr^{-n+k+1}(Va (k)-f(k)), \quad n\geq 2.
    \]
    Similarly we have
    \[
        a (n)-\Gr^{-1} a (n-1)=\Gr^{n-1}a (1)+\frac{\Gb^2}{\Gg}\sum_{k=1}^{n-1}\Gr^{n-k-1}(Va (k)-f(k)), \quad n\geq 2.
    \]
    Combining above two identities together, we obtain the following relation:
    \begin{equation}
        \label{eq_a_general}
        \begin{aligned}
            a (n)=&\,\frac{\Gr^{-n}}{\Gr^{-1}-\Gr}\left(a (1)+\frac{\Gb^2}{\Gg}\sum_{k=1}^{n-1}\Gr^k (Va (k)-f(k))\right) \\
            &-\frac{\Gr^n}{\Gr^{-1}-\Gr}\left(a (1)+\frac{\Gb^2}{\Gg}\sum_{k=1}^{n-1}\Gr^{-k} (Va (k)-f(k))\right)
        \end{aligned}
    \end{equation}
    for $n\geq 2$.

    Since $a \in l^\infty (\Nbb, \Rbb)$, it must hold that
    \[
        a (1)+\frac{\Gb^2}{\Gg}\sum_{k=1}^\infty \Gr^k (Va (k)-f(k))=0.
    \]
    Thanks to the relation \eqref{eq_rho_beta_rel} between $\Gr$ and $\Gb$, \eqref{eq_a_general} becomes
    \begin{equation}
        \begin{aligned}
            a (n)=&\,-\frac{r\Gb\Gr^{-n}}{\Gg}\sum_{k=n}^\infty \Gr^k (Va (k)-f(k)) \\
            &+\frac{r\Gb\Gr^n}{\Gg}\left(\sum_{k=1}^\infty \Gr^k (Va (k)-f(k))-\sum_{k=1}^{n-1}\Gr^{-k} (Va (k)-f(k))\right).
        \end{aligned}
    \end{equation}
 Employing the formula
    \[
        \nabla_\mathrm{D}(\xi \eta)(n)=(\nabla_\mathrm{D} \xi)(n)\eta (n)+\xi (n+1)(\nabla_\mathrm{D} \eta)(n)
    \]
    for $\xi, \eta: \Nbb\to \Cbb$, we then have
    \begin{align*}
        &\nabla_\mathrm{D} a (n) \\
        &=-\frac{r\Gb\Gr^{-n-1}}{\Gg}\left((1-\Gr)\sum_{k=n}^\infty \Gr^k (Va (k)-f(k))+\Gr^n (Va (n)-f(n))\right) \\
        &\quad+\frac{r\Gb\Gr^n}{\Gg}\biggl((\Gr-1)\sum_{k=1}^\infty \Gr^k (Va (k)-f(k)) \\
        &\quad-(\Gr-1)\sum_{k=1}^{n-1}\Gr^{-k} (Va (k)-f(k))-\Gr^{-n+1}(Va (n)-f(n))\biggr) \\
        &=-\frac{r\Gb}{\Gg}(1-\Gr)\biggl(\Gr^{-n-1}\sum_{k=n}^\infty \Gr^k (Va (k)-f(k))+\Gr^n\sum_{k=1}^\infty \Gr^k (Va (k)-f(k)) \\
        &\quad+\Gr^n\sum_{k=1}^{n-1}\Gr^{-k} (Va (k)-f(k))\biggr)
        -\frac{r\Gb}{\Gg}(\Gr^{-1}+\Gr) (Va (n)-f(n))
    \end{align*}
    for $n\geq 2$.

    Terms appearing after the last equality sign are estimated as follows:
    \begin{align*}
        \left|\sum_{k=n}^\infty \Gr^k (Va (k)-f(k))\right|
        \leq \sum_{k=n}^\infty \Gr^k (Va (k)+f(k))
        \leq 6\sum_{k=n}^\infty \Gr^{2k}
        =\frac{6\Gr^{2n}}{1-\Gr^2}
    \end{align*}
    for $n\in \Nbb$ and
    \begin{align*}
        \left|\sum_{k=1}^{n-1}\Gr^{-k} (Va (k)-f(k))\right|
        \leq \sum_{k=1}^{n-1}\Gr^{-k} (Va (k)+f(k))
        \leq 6\sum_{k=1}^{n-1}1=6(n-1)
    \end{align*}
    for $n\geq 2$. Combining these estimates with \eqref{eq_Vaf}, we have
    \begin{align*}
        |\nabla_\mathrm{D} a (n)|
        \leq &\,\frac{r\beta}{\gamma}(1-\Gr)\left(\Gr^{-n-1}\frac{6\Gr^{2n}}{1-\Gr^2}+\Gr^n\frac{6\Gr^2}{1-\Gr^2}
        +6\Gr^n (n-1)\right) \\
        &+\frac{6r\beta}{\gamma}(\Gr^{-1}+\Gr) \Gr^n \\
        \leq &\, \frac{6r\beta}{\gamma}((1-\Gr)(n-1)+4)\Gr^{n-1},
    \end{align*}
    which is \eqref{eq_diff_estimate}.
\end{proof}

We now estimate $S_2$, $S_3$ and $S_4$.

\begin{lemm}\label{lemm_S2}
    It holds that
    \begin{equation}
        \label{eq_S2}
        |S_2|\leq 3\beta+\frac{18\beta (r+\Ge)}{\gamma}
    \end{equation}
    for all $r, \Ge, \gamma>0$ and $\zeta\in A$.
\end{lemm}

\begin{proof}
    From \eqref{eq_lap_estimate} in Theorem \ref{theo_diff_estimate} and
    \begin{equation}\label{eq_rho_beta_est}
        2r\leq \frac{\Gb}{1-\Gr}=\frac{4r+\Ge+\sqrt{\Ge (4r+\Ge)}}{2}\leq 3r+\Ge,
    \end{equation}
    we obtain
    \begin{align*}
        \left|\sum_{n=2}^\infty \frac{\Gr^n}{1+\Gr^{2n}}\lap_\mathrm{D}a (n)\Gz^n\right|
        &\leq 2(1-\Gr)^2 \sum_{n=2}^\infty (\beta n+r)\Gr^{n-1}+\frac{6\beta^2}{\gamma}\sum_{n=2}^\infty \Gr^n \\
        &=2\beta \Gr (2-\Gr)+2r\Gr (1-\Gr)+\frac{6\beta^2}{\gamma (1-\Gr)} \\
        &\leq 3\beta+\frac{18\beta (r+\Ge)}{\gamma}. \qedhere
    \end{align*}
\end{proof}

\begin{lemm}
    \label{lemm_S34}
    It holds that
    \begin{equation}
        \label{eq_S34}
        |S_3|\leq \frac{36r\beta}{\gamma}, \quad
        |S_4| \leq 6\beta C_D+\frac{36r\beta}{\gamma}
    \end{equation}
    for all $r, \Ge, \Gg>0$ and $\Gz\in A$.
\end{lemm}

\begin{proof}
    We only derive the estimate for $|S_4|$. That for $|S_3|$ is similarly derived.

    We have
    \[
        \left(\nabla_\mathrm{D} \frac{\Gr^n}{1+\Gr^{2n}}\right)(n)=-\frac{(1-\Gr)\Gr^n (1-\Gr^{2n+1})}{(1+\Gr^{2n+2})(1+\Gr^{2n})}
    \]
    for $n\geq 1$ and thus
    \begin{equation}
        \label{eq_grad_rho}
        \left|\left(\nabla_\mathrm{D} \frac{\Gr^n}{1+\Gr^{2n}}\right)(n)\right|
        \leq (1-\Gr)\Gr^n
    \end{equation}
    for $n\geq 1$.

    From Corollary \ref{coro_agn_estimate_rough}, \eqref{eq_diff_estimate} and \eqref{eq_grad_rho}, we obtain
    \begin{align*}
        |S_4|
        &\leq (1-\Gr)\sum_{n=2}^\infty \Gr^{n-1}|\nabla_\mathrm{D} a (n-1)|\Gr^{-n} \\
        &= (1-\Gr)\Gr^{-1}|\nabla_\mathrm{D} a (1)|+(1-\Gr)\Gr^{-1}\sum_{n=2}^\infty |\nabla_\mathrm{D} a (n)| \\
        &\leq 2C \beta (1-\Gr)(1+2\Gr)+\frac{6r\beta}{\gamma}(1-\Gr)\sum_{n=2}^\infty ((1-\Gr)n+4)\Gr^{n-2} \\
        &=2C \beta (1-\Gr)(1+2\Gr)+\frac{6r\beta}{\gamma}(6-\Gr)
        \leq 6C \beta +\frac{36r\beta}{\gamma}. \qedhere
    \end{align*}
\end{proof}

Putting all the estimates above together, we have the following theorem.

\begin{theo}
    \label{theo_existence_two_disks}
    Let $U$ be the function constructed in this section and let $u = U \circ T$. It holds that
    \begin{equation}
        \label{eq_gradient_estimate}
        \|\nabla u\|_{L^\infty (\Rbb^2\setminus \overline{D})}\leq 180\left(1+\frac{1+|1-\gamma/r|}{1-\Gr+\gamma/r}+\frac{r+\Ge}{\gamma}\right)
    \end{equation}
    for all $r, \Ge, \Gg>0$.
\end{theo}

\begin{proof}
By \eqref{eq_grad_abel}, Lemma \ref{lemm_estimate_beginning}, Lemma \ref{lemm_S0} and Lemma \ref{lemm_S1}, we obtain
    \begin{equation}\label{eq_grad_wip}
        |(\Gz-1)^2\p_\Gz F (\Gz)|
        \leq 10\Gb C_D \left|1-\frac{\Gg}{r}\right|+8\Gb+6\Gb C_D+|S_2|+|S_3|+|S_4|.
    \end{equation}
    By Lemma \ref{lemm_S2} and Lemma \ref{lemm_S34}, we have
    \begin{align}
        |S_2|+|S_3|+|S_4|
        &\leq \left(3\beta+\frac{18\beta (r+\Ge)}{\gamma}\right)+\frac{36\beta r}{\gamma}
        +\left(6\beta C_D+\frac{36\beta r}{\gamma}\right) \nonumber \\
        &=3\beta+6\beta C_D+\frac{18\beta (5r+\Ge)}{\gamma}.
        \label{eq_S234_2}
    \end{align}
    Combining \eqref{eq_grad_wip} and \eqref{eq_S234_2}, we obtain
    \begin{align*}
        |(\Gz-1)^2 \p_\Gz F (\Gz)|
        \leq &\, \left(10\beta C_D \left|1-\frac{\gamma}{r}\right|+8\beta+6\beta C_D\right) \\
        &+\left(3\beta+6\beta C_D+\frac{18\beta (5r+\Ge)}{\gamma}\right) \\
        \leq &\, 12\beta \left(1+C_D \left(1+\left|1-\frac{\gamma}{r}\right|\right)\right)
        +\frac{90\beta (r+\Ge)}{\gamma} \\
        =&\, 12\beta + \frac{12\beta (1+|1-\gamma/r|)}{1-\Gr+\gamma/r}+\frac{90\beta (r+\Ge)}{\gamma}.
    \end{align*}
    We combine this estimate with \eqref{eq_grad_u} to obtain the conclusion.
\end{proof}

Theorem \ref{mainLC} now can be proved easily.
\begin{proof}[Proof of Theorem \ref{mainLC}]
That the solution $u$ belongs to the space $h+H^1_\mathrm{loc}(\Rbb^2\setminus D)$ is clear since $u$ is smooth up to $\p D$. Uniqueness of the solution is already proved.

If $\Gg \ge \Gg_0$ for some $\Gg_0>0$, then the quantity on the right hand side of \eqref{eq_gradient_estimate} are bounded regardless of the distance between inclusions. Thus \eqref{LC-C_gradient_estimate} follows.
\end{proof}

\section{Proof of Theorem \ref{mainHC}}\label{sect_HC}

We just show that two problems \eqref{LCC} and \eqref{HCI} in two dimensions are dual problems in the sense that the solution of the one problem is a harmonic conjugate of the solution to the other problem.

Let $u$ be the solution to \eqref{LCC} for a given entire harmonic function $h$. Let $\widetilde{h}$ be a harmonic conjugate of $h$ in $\Rbb^2$. Thanks to the third condition in \eqref{LCC}, $u$ has the harmonic conjugate $v$ in $\Rbb^2\setminus \overline{D}$ such that $v(x) - \widetilde{h}(x) = O(|x|^{-1})$. By taking tangential derivative $\p_\Gj$ of the both sides of the second line in \eqref{LCC}, we have from the Cauchy-Riemann equations that
\begin{align}
- \p_\Gv v|_+ = \Gg \p_\Gj^2 v .
\end{align}
Since $\p_\Gj^2 v= \GD_S v$, $v$ is the solution to \eqref{HCI} with $\Ga=\Gg$ and $h$ replaced with $\widetilde{h}$.

The converse can be seen similarly. Let $v$ be the solution to \eqref{HCI} for a given entire harmonic function $h$ and let $\widetilde{h}$ be a harmonic conjugate of $h$ in $\Rbb^2$. By the second condition of \eqref{HCI}, we have
$$
    \int_{\p D_j} \p_{\Gv} v |_+\; \df\Gs= -\Ga \int_{\p D_j} \GD_S v \; \df\Gs =0, \quad j=1,2.
$$
Thus $v$ admits a harmonic conjugate, denoted by $u$, in $\Rbb^2\setminus \overline{D}$ such that $u(x) - \widetilde{h}(x) = O(|x|^{-1})$. This $u$ is the solution to \eqref{LCC} with $h$, $\Gg$ replaced by $\widetilde{h}$, $\Ga$, respectively.

\section*{Discussion}\addcontentsline{toc}{section}{\protect\numberline{}Discussion}
In this paper we investigate stress in a system where inclusions are bonded to the matrix by the biological bonding. As a starting point of such investigation, we consider the simplest geometry of circular inclusions. We also dealt with inclusions with imperfect bonding conditions which are approximations of the core-shell structure. We show that the stress (the gradient of the solution) stays bounded even if two inclusions are arbitrarily close as long as the imperfect bonding parameter is finite. It would be quite important and challenging to extend this result to more general case: more general geometry in two and three dimensions and to a more general bonding such as the core-shell structure of membranes.

\section*{Acknowledgement} Authors thank Tao Li for helpful discussions.

\bibliography{FJKL_finite_stress}

\begin{thebibliography}{10}

\bibitem{AKLLL}
H.~Ammari, H.~Kang, H.~Lee, J.~Lee, and M.~Lim.
\newblock Optimal estimates for the electric field in two dimensions.
\newblock {\em J. Math. Pures Appl. (9)}, 88(4):307--324, 2007.

\bibitem{AKL}
H.~Ammari, H.~Kang, and M.~Lim.
\newblock Gradient estimates for solutions to the conductivity problem.
\newblock {\em Math. Ann.}, 332(2):277--286, 2005.

\bibitem{ABR01}
S.~Axler, P.~Bourdon, and W.~Ramey.
\newblock {\em Harmonic function theory}, volume 137 of {\em Graduate Texts in
  Mathematics}.
\newblock Springer-Verlag, New York, second edition, 2001.

\bibitem{BLY}
E.~S. Bao, Y.~Y. Li, and B.~Yin.
\newblock Gradient estimates for the perfect conductivity problem.
\newblock {\em Arch. Ration. Mech. Anal.}, 193(1):195--226, 2009.

\bibitem{BM1999}
Y.~Benveniste and T.~Miloh.
\newblock Neutral inhomogeneities in conduction phenomena.
\newblock {\em J. Mech. Phys. Solids}, 47(9):1873--1892, 1999.

\bibitem{DL}
H.~Dong and H.~Li.
\newblock Optimal estimates for the conductivity problem by {G}reen's function
  method.
\newblock {\em Arch. Ration. Mech. Anal.}, 231(3):1427--1453, 2019.

\bibitem{DLY22}
H.~Dong, Y.~Y. Li, and Z.~Yang.
\newblock Gradient estimates for the insulated conductivity problem: The
  non-umbilical case.
\newblock arXiv:2203.10081 [math.AP].

\bibitem{DLY21}
H.~Dong, Y.~Y. Li, and Z.~Yang.
\newblock Optimal gradient estimates of solutions to the insulated conductivity
  problem in dimension greater than two.
\newblock to appear in \textit{J. Eur. Math. Soc.} arXiv:2110.11313 [math.AP].

\bibitem{Evans10}
L.~C. Evans.
\newblock {\em Partial differential equations}, volume~19 of {\em Graduate
  Studies in Mathematics}.
\newblock American Mathematical Society, Providence, RI, second edition, 2010.

\bibitem{Hashin62}
Z.~Hashin.
\newblock The elastic moduli of heterogeneous materials.
\newblock {\em J. Appl. Mech.}, 29:143--150, 1962.

\bibitem{Ji-Kang22}
Y.-G. Ji and H.~Kang.
\newblock Spectrum of the {N}eumann-{P}oincar\'{e} operator and optimal
  estimates for transmission problems in the presence of two circular
  inclusions.
\newblock {\em Int. Math. Res. Not. IMRN}, 2023(9):7638--7685, 2023.

\bibitem{Kang23}
H.~Kang.
\newblock Quantitative analysis of field concentration in presence of closely
  located inclusions of high contrast.
\newblock In {\em Proceedings ICM 2022}, volume~7, pages 5680--5701, 2023.

\bibitem{KL19}
H.~Kang and X.~Li.
\newblock Construction of weakly neutral inclusions of general shape by
  imperfect interfaces.
\newblock {\em SIAM J. Appl. Math.}, 79(1):396--414, 2019.

\bibitem{Keller-JAP-63}
J.~B. Keller.
\newblock Conductivity of a medium containing a dense array of perfectly
  conducting spheres or cylinders or nonconducting cylinders.
\newblock {\em J. Appl. Phys.}, 34(4):991--993, 1963.

\bibitem{LN}
Y.~Y. Li and L.~Nirenberg.
\newblock Estimates for elliptic systems from composite material.
\newblock {\em Comm. Pure Appl. Math.}, 56(7):892--925, 2003.

\bibitem{LV}
Y.~Y. Li and M.~Vogelius.
\newblock Gradient estimates for solutions to divergence form elliptic
  equations with discontinuous coefficients.
\newblock {\em Arch. Ration. Mech. Anal.}, 153(2):91--151, 2000.

\bibitem{LY-CPDE-09}
M.~Lim and K.~Yun.
\newblock Blow-up of electric fields between closely spaced spherical perfect
  conductors.
\newblock {\em Comm. Partial Differential Equations}, 34(10-12):1287--1315,
  2009.

\bibitem{Torquato-Rintoul95}
S.~Torquato and M.~Rintoul.
\newblock Effect of the interface on the properties of composite media.
\newblock {\em Phys. Rev. Lett}, 75(22):4067, 1995.

\bibitem{Yun}
K.~Yun.
\newblock Estimates for electric fields blown up between closely adjacent
  conductors with arbitrary shape.
\newblock {\em SIAM J. Appl. Math.}, 67(3):714--730, 2007.

\bibitem{Yun16}
K.~Yun.
\newblock An optimal estimate for electric fields on the shortest line segment
  between two spherical insulators in three dimensions.
\newblock {\em J. Differential Equations}, 261(1):148--188, 2016.

\end{thebibliography}
\bibliographystyle{abbrv}

\end{document}